\documentclass{article}

\usepackage{indentfirst}
\usepackage{amsmath,amsfonts,amsthm,amssymb}
\usepackage{mathrsfs}
\usepackage{amscd}

\def\co{\colon\thinspace}
\DeclareMathAlphabet{\mathsfsl}{OT1}{cmss}{m}{sl}

\newtheorem{thm}{Theorem}[section]
\newtheorem{lem}[thm]{Lemma}
\newtheorem{cor}[thm]{Corollary}
\newtheorem{prop}[thm]{Proposition}

\theoremstyle{definition}
\newtheorem{defn}[thm]{Definition}

\newtheorem{rem}[thm]{Remark}

\newcommand{\HFK}{\widehat{HFK}}
\newcommand{\on}{\operatorname}

\newcommand{\Spinc}{\on{Spin}^c}

\newcommand{\Z}{\mathbb{Z}}

\newcommand{\Q}{\mathbb{Q}}

\newcommand\goth[1]{\mathfrak{#1}}
\newcommand{\s}{\goth{s}}

\newcommand{\rank}{\text{rank}}

\begin{document}

\title{On mapping cones of Seifert fibered surgeries}

\author{{\large Zhongtao WU}\\{\normalsize Department of Mathematics, Caltech, MC 253-37}\\
{\normalsize 1200 E California Blvd, Pasadena, CA
91125}\\{\small\it Emai\/l\/:\quad\rm zhongtao@caltech.edu}}

\date{}
\maketitle

\begin{abstract}
Using the mapping cone of a rational surgery, we give several obstructions for Seifert fibered surgeries, including obstructions on the Alexander polynomial, the knot Floer homology, the surgery coefficient and the Seifert and four-ball genus of the knot.  These generalize the corresponding results in \cite{KMOS}\cite{OSzSeifert}.
\end{abstract}

\section{Introduction}

Seifert fibered spaces comprise a special, yet broad class of three-manifolds. A Seifert fibered space is a three-manifold together with a ``nice'' decomposition as a disjoint union of circles.  They account for all compact oriented manifolds in six of the eight Thurston geometries of the geometrization conjecture.  It is a meaningful question to characterize knots $K$ in the three-sphere so that certain surgery along $K$ can yield a Seifert fibered rational homology sphere.  This question has received considerable attention recently \cite{Dean}\cite{EM}\cite{G}.

In \cite{OSzSeifert}, certain obstructions were established that comes from the Heegaard Floer homology of \cite{OSzAnn1} and the related knot invariant defined in \cite{OSzKnot} and \cite{RasThesis}.  Ozsv\'ath and Szab\'o, among other things, proved that all the non-zero torsion coefficients $t_i(K)$ (to be defined soon) have the same sign if there is an integer $q\neq 0$ for which $S^3_{1/q}(K)$ is Seifert fibered.  In \cite{KMOS}, Kronheimer, Mrowka, Ozsv\'ath and Szab\'o used the monopole Floer homology to give an obstruction on the degree of Alexander polynomial.  More precisely, they proved that if $K$ is a knot whose Alexander polynomial $\Delta_K(T)$ has degree strictly less than its Seifert genus, then there is no rational number $p/q \geq0$ such that $S^3_{p/q}(K)$ is a positively oriented Seifert fibered space.  In both references \cite{OSzSeifert}\cite{KMOS}, it was also shown that when the Seifert genus $g>1$, $S^3_{1/q}(K)$ is never a negatively Seifert fibered space for any $q>0$ if $\Delta_K(T)$ has degree strictly less than $g$.  The aim of the present article is to generalize these results to all rational $p/q$-surgeries and Seifert fibered spaces of both orientations.

The Seifert fibered spaces to be considered in this article are all rational homology spheres.  Such a manifold can be realized as the boundary of a four-manifold $W(\Gamma)$ obtained by plumbing two-spheres according to a weighted tree $\Gamma$.  We define the orientation of a Seifert fibered space from the intersection form.

\begin{defn}
Suppose that $\Gamma$ is a weighted tree which has either negative-definite or negative-semi-definite intersection form.  Then, we say that the induced orientation $-\partial W(\Gamma)$ is a positive Seifert orientation.

\end{defn}

Note that if $Y$ is a Seifert fibered space with $b_1(Y)=0$ then at least one of $+Y$ or $-Y$ has a positive Seifert orientation.  Moreover, either orientation on any lens spaces is a positive Seifert orientation.  Our first obstruction of Seifert fibered surgeries can be expressed in terms of the Alexander polynomial.

\begin{thm}\label{PSFS}
Let $K\subset S^3$ be a knot in the three-sphere.  Write its symmetrized Alexander polynomial as
$$\Delta_K(T)=a_0+\sum_{i>0}a_i(T^i+T^{-i}),$$
and let
\begin{equation}\label{torsion}
t_i(K)=\sum_{j=1}^{\infty}ja_{|i|+j}
\end{equation}
denote the torsion coefficients of the knot.  Then, if there is a rational number $p/q>0$ for which $S^3_{p/q}(K)$ is positively Seifert fibered, all the torsion coefficients $t_i(K)$ are non-negative; if there is a rational number $0<p/q<3$ for which $S^3_{p/q}(K)$ is negatively Seifert fibered, then for all $i>0$ the torsion coefficients $t_i(K)$ are non-positive.

\end{thm}

We have also a Seifert fibered surgery obstruction, which can be stated in terms of the knot Floer homology.  The knot Floer homology determines the genus of the knot
\begin{equation}\label{genus}
g(K)=\mathrm{max}\{i|\HFK (K,i)\neq 0\},
\end{equation}
and the Alexander polynomial:
\begin{equation}\label{Alexander}
\Delta_K(T)=\sum_{i\in\Z}\chi(\HFK(K,i))\cdot T^i.
\end{equation}

\begin{thm}\label{knotobst}
Let $K\subset S^3$ be a knot with genus $g$.  If there is a rational number $p/q>0$ such that $S^3_{p/q}(K)$ is a positively oriented Seifert fibered space, then $\HFK(S^3,K,g)$ is trivial in odd degrees (and non-trivial in even degrees).  If there is a rational number $p/q>0$ such that $S^3_{p/q}(K)$ is a negatively oriented Seifert fibered space, and the genus of the knot $g>1$ and $2g-1>p/q$, then $\HFK(S^3,K,g)$ is trivial in even degrees (and non-trivial in odd degrees).
\end{thm}

This has the following corollary (compare \cite{OSzSeifert}\cite{KMOS}).

\begin{cor}\label{surgcoeff}
If $\mathrm{deg}\,\Delta_K<g(K)$, then for $p/q >0$, $S^3_{p/q}(K)$ is never a positively oriented Seifert fibered space.  If in addition, $g>1$ and $2g-1>|p/q|$, then no $p/q$-surgery along $K$ is Seifert fibered.
\end{cor}

\begin{proof}
Note that $S^3_{p/q}(K)=-S^3_{-p/q}(m(K))$, where $m(K)$ denotes the mirror of $K$.  The corollary follows immediately from the Euler characteristic relation (\ref{Alexander}) and Theorem \ref{knotobst}.
\end{proof}

In \cite {Ichi}, Ichihara showed that for a hyperbolic knot $K$, if $|p/q|>3\cdot2^{7/4}g\approx 10.1g$, then $S^3_{p/q}(K)$ is a hyperbolic three-manifold (thus, is not a Seifert fibered space).  Corollary \ref{surgcoeff} basically supplies a bound on the surgery coefficients from the other end when the condition on the degree of the Alexander polynomial $\mathrm{deg}\,\Delta_K<g(K)$ is satisfied.  Concerning the surgeries yielding lens spaces, Goda and Teragaito in \cite{Goda} conjectured that $2g+8\leq |p/q|\leq 4g-1$ if a Dehn surgery with slope $p/q$ on a hyperbolic knot in $S^3$ yields a lens space.  This conjecture is now settled following works of Kronheimer-Mrowka-Ozsv\'ath-Szab\'o \cite{KMOS}, Rasmussen \cite{RasLens} and Greene \cite{Greene}.  It is natural to compare Corollary \ref{surgcoeff} with \cite[Corollary 1.4]{OSzRatSurg}, which says that $S^3_{p/q}(K)$ is not an $L$-space if $2g-1>|p/q|$.

Finally, we present two Seifert fibered surgery obstructions in terms of the four-ball genus of a knot.

\begin{thm}\label{fourballgenus}
Let $K\subset S^3$ be a knot with genus $g>1$.  If $\mathrm{deg}\;\Delta_K<g$ and the four-ball genus $g^*(K)<g$, then $K$ does not admit a Seifert fibered surgery.

\end{thm}

As an application of the theorem, we consider the family of Kinoshita-Terasaka knots $KT_{r,n}$ with $|r|\geq 2$ and $n\neq 0$ that was also used for illustrations in \cite{OSzSeifert}.  These knots all have trivial Alexander polynomial, has genus $|r|$ (\cite{Gabai}\cite{OSzMutation}) and is slice, i.e. $g^*(K)=0$.  Thus, Theorem \ref{fourballgenus} applies and these knots do not admit Seifert fibered surgeries.

\begin{thm}\label{sliceknot}
Let $K\subset S^3$ be a slice knot.  If there are both positive and negative torsion coefficients $t_i(K)$, then $K$ does not admit a Seifert fibered surgery.
\end{thm}

We checked the list of all slice knots with fewer than twelve crossings. By applying the above obstructions (Theorem \ref{fourballgenus},\ref{sliceknot}), we found that among the 76 alternating slice knots, only the knot $6_1$ and $10_3$ could possibly admit a Seifert fibered surgery, and the knot $6_1$, also known as the Stevedore's knot, indeed admitted one.  Also, among the 81 non-alternating slice knots, there are only 17 of them could potentially admit a Seifert fibered surgery.

\vspace{5pt}\noindent{\bf Acknowledgements.} I wish to thank Zolt\'an Szab\'o, Yi Ni, Paul Kirk and Charles Livingston for interesting discussions on this project.  I am especially grateful to Yi for inspiring me the idea of looking at the size of the surgery coefficient $p/q$, as well as his help on several technical difficulties in this article.  The author is supported by a Simons Postdoctoral Fellowship.

\section{Rational surgeries and the mapping cones}

In this section, we recall the rational surgery formula of Ozsv\'ath and Szab\'o \cite{OSzRatSurg}.  For simplicity, we will use $\mathbb{F}_2=\Z/2\Z$ coefficients for Heegaard Floer homology throughout this paper.

Given a knot $K$ in an integer homology sphere $Y$. Let $C=CFK^{\infty}(Y,K)$ be the knot Floer chain complex of $(Y,K)$. There are chain complexes
$$A^+_k=C\{i\ge0 \text{ or }j\ge k\},\quad k\in\mathbb Z$$
and $B^+=C\{i\ge0\}\cong CF^+(Y)$. As in \cite{OSzIntSurg}, there are chain maps
$$v_k,h_k\co A^+_k\to B^+.$$

Let $$\mathbb A_i^+=\bigoplus_{s\in\mathbb Z}(s,A^+_{\lfloor\frac{i+ps}q\rfloor}(K)),\mathbb B_{i}^+=\bigoplus_{s\in\mathbb Z}(s,B^+).$$
For a given rational number $p/q\neq 0$, define maps
$$v^+_{\lfloor\frac{i+ps}q\rfloor}\co (s,A^+_{\lfloor\frac{i+ps}q\rfloor}(K))\to (s,B^+),\quad h^+_{\lfloor\frac{i+ps}q\rfloor}\co (s,A^+_{\lfloor\frac{i+ps}q\rfloor}(K))\to (s+1,B^+).$$
Adding these up, we get a chain map
$$D_{i,p/q}^+\co\mathbb A_i^+\to \mathbb B_i^+,$$
with
$$D_{i,p/q}^+\{(s,a_s)\}_{s\in\mathbb Z}=\{(s,b_s)\}_{s\in\mathbb Z},$$
where
$$b_s=v^+_{\lfloor\frac{i+ps}q\rfloor}(a_s)+h^+_{\lfloor\frac{i+p(s-1)}q\rfloor}(a_{s-1}).$$

\begin{thm}[Ozsv\'ath--Szab\'o]\label{OSmappingcone}
Let $\mathbb X^+_{i,p/q}$ be the mapping cone of $D_{i,p/q}^+$, then there is a relatively graded isomorphism of groups
\begin{equation}\label{rationalsurgery}
H_*(\mathbb X^+_{i,p/q})\cong HF^+(Y_{p/q}(K),i).
\end{equation}
\end{thm}

The mapping cone can be as well applied to the zero-surgery of a knot.  Specifically, the integer surgeries long exact sequence (of \cite[Theorem 10.19]{OSzAnn2}) shows that $HF^+(Y_0(K),i)$ is identified with the homology of the mapping cone of
\begin{equation}\label{zerosurgery}
v^+_i+h^+_i\co A^+_i\to B^+.
\end{equation}

\begin{rem}
Let $$\mathfrak D_{i,p/q}^+\co H_*(\mathbb A_i^+)\to H_*(\mathbb B_i^+)$$
be the map induced by $D_{i,p/q}^+$ on homology. Then the mapping cone of $\mathfrak D_{i,p/q}^+$ is quasi-isomorphic to $\mathbb X^+_{i,p/q}$.
By abuse of notation, we do not distinguish $A_k^+,\mathbb A_i^+,B^+,\mathbb B_i^+$ from their homology, and do not distinguish $D_{i,p/q}^+$ from $\mathfrak D_{i,p/q}^+$.
\end{rem}

\begin{rem}\label{gradingshift}
Although not explicitly exhibited in \cite{OSzIntSurg}\cite{OSzRatSurg}, we can assign a $\Z/2\Z$-grading on $\mathbb X^+_{i,p/q}$ so that the isomorphism (\ref{rationalsurgery}) respects $\Z/2\Z$-grading - by keeping the original $\Z/2\Z$-grading on $\mathbb A_i^+$ and reversing the original $\Z/2\Z$-grading on $\mathbb B_i^+$.  Similarly, we can assign a $\Z/2\Z$-grading on the mapping cone of (\ref{zerosurgery}) so that the isomorphism to $HF^+(Y_0(K),i)$ respects $\Z/2\Z$-grading - by reversing the original $\Z/2\Z$-grading on $A_i^+$ and keeping the original $\Z/2\Z$-grading on $B_i^+$.

\end{rem}

For a rational homology three-sphere $Y$ with a $\Spinc$ structure $\s$, $HF^+(Y,\s)$ can be decomposed as the direct sum of two groups: The first group is the image of $HF^\infty(Y,\s)$ in $HF^+(Y,\s)$, whose minimal absolute $\mathbb{Q}$ grading is an invariant of $(Y,\s)$ and is denoted by $d(Y,\s)$, the {\it correction term} \cite{OSzAbGr}; the second group is the quotient modulo the above image and is denoted by $HF_{\mathrm{red}}(Y,\s)$.  Altogether, we have $$HF^+(Y,\s)=\mathcal{T}^+_{d(Y,\s)}\oplus HF_{\mathrm{red}}(Y,\s).$$

For a knot $K\subset S^3$, let $A_k^T=U^nA^+_k$ for $n\gg0$, then $A_k^T\cong\mathcal T^+$.
Let $D_{i,p/q}^T$ be the restriction of $D_{i,p/q}^+$ on
\begin{equation}\label{freepart}
\mathbb A_i^T=\bigoplus_{s\in\mathbb Z}(s,A^T_{\lfloor\frac{i+ps}q\rfloor}(K)).
\end{equation}

Since $v^+_k,h^+_k$ are isomorphisms at sufficiently high gradings and are $U$--equivariant, $v^+_k|A_k^T$ is modeled on multiplication by $U^{V_k}$ and $h^+_k|A_k^T$ is modeled on multiplication by $U^{H_k}$, where $V_k,H_k\ge0$. We list some of the properties of $V_k$ and $H_k$ here, and refer their proofs to \cite{NiWu}\cite{RasThesis}.

\begin{lem}\label{lem:VHMono}\cite{RasThesis}
The nonnegative integers $V_k$ and $H_k$ satisfy the following properties:
\begin{enumerate}
\item $V_k=H_{-k}$.
\item For $k>0$, we have $V_k+1\geq V_{k-1}\geq V_k$.
\item $V_k \leq \mathrm{max}\{0, g^*(K)-k\}$, where $g^*(K)$ is the slice genus of $K$.
\end{enumerate}

\end{lem}

It is clear that $V_k=0$ when $k\ge g^*$ and $H_k=0$ when $k\le -g^*$; also $V_k\to +\infty$ as $k\to -\infty$ and $H_k\to +\infty$ as $k\to+\infty$.

\begin{lem}\label{lem:V0H0}
$V_0=H_0$. Moreover, $V_k> H_k$ if $k<0$ and $V_k< H_k$ if $k>0$.
\end{lem}
\begin{proof}
If $$(\Sigma,\mbox{\boldmath${\alpha}$},
\mbox{\boldmath$\beta$},w,z)$$ is a doubly pointed Heegaard
diagram for $(S^3,K)$, then
$$(-\Sigma,\mbox{\boldmath${\beta}$},
\mbox{\boldmath$\alpha$},z,w)$$
is also a Heegaard diagram for $(S^3,K)$. Hence the roles of $i,j$ can be interchanged. It follows that $v^+_0$ is equivalent to $h^+_0$, hence $V_0=H_0$.  For the second part, note that $HF^+(S^3_0(K),i)=HF^+_{\mathrm{red}}(S^3_0(K),i)$ for a non-torsion $\Spinc$ structure $i$.  By the mapping cone formula (\ref{zerosurgery}), we must have $V_i\neq H_i$ for all non-zero integers $i$.  The inequalities then follow from $V_0=H_0$ and Lemma \ref{lem:VHMono}.

\end{proof}

\begin{lem}\label{lem:>0Surj}
Suppose $p/q>0$. Then the map $D_{i,p/q}^T$ is surjective.
\end{lem}
The exact same lemma is proved in \cite[Lemma 2.9]{NiWu}.  Nevertheless, we include the proof here so that the reader may have an opportunity to get used to our notations.

\begin{proof}
Let $0\le i\le p-1$. Suppose
$$\mbox{\boldmath${\eta}$}=\{(s,\eta_s)\}_{s\in\mathbb Z}\in\mathbb B^+_i.$$
Let
$$\xi_{-1}=U^{-H_{\lfloor\frac{i+p(-1)}q\rfloor}}\eta_{0},\quad \xi_{0}=0.$$
For other $s$, let
$$
\xi_{s}=\bigg\{
\begin{array}{ll}
U^{-V_{\lfloor\frac{i+ps}q\rfloor}}(\eta_{s}-U^{H_{\lfloor\frac{i+p(s-1)}q\rfloor}}\xi_{s-1}),&\text{if }s>0,\\
U^{-H_{\lfloor\frac{i+ps}q\rfloor}}(\eta_{s+1}-U^{V_{\lfloor\frac{i+p(s+1)}q\rfloor}}\xi_{s+1}),&\text{if }s<-1.
\end{array}
$$
By the definition of direct sum, $\eta_s=0$ when $|s|\gg0$. Using the facts that
$$H_{\lfloor\frac{i+p(s-1)}q\rfloor}-V_{\lfloor\frac{i+ps}q\rfloor}\to+\infty, \text{ as }s\to+\infty,$$
and
$$V_{\lfloor\frac{i+p(s+1)}q\rfloor}-H_{\lfloor\frac{i+ps}q\rfloor}\to+\infty, \text{ as }s\to-\infty,$$
we see that $\xi_s=0$ when $|s|\gg0$. So $\mbox{\boldmath${\xi}$}=\{(s,\xi_s)\}_{s\in\mathbb Z}\in\mathbb A_i^T$.
Clearly
$$D_{i,p/q}^T(\mbox{\boldmath${\xi}$})=\mbox{\boldmath${\eta}$}.$$
\end{proof}

\section{Seifert fibered surgeries}

For a three-manifold $Y=-\partial W(\Gamma)$, where $\Gamma$ is a negative-definite graph with at most one bad point, $HF^+(Y)$ can be explicitly calculated in terms of the graph $\Gamma$.  The part of that calculation which we will use in the present paper can be summarized as follows:

\begin{thm}\label{OSzPlumb}\cite{OSzPlumb}
Let $\Gamma$ be a negative-definite graph with at most one bad point.  Then, $HF^+(-\partial W(\Gamma))$ is supported in even dimensions.  Moreover, if $\Gamma$ has no bad points, then $HF^+_{\mathrm{red}}(-\partial W(\Gamma))=0$.

\end{thm}

\subsection{Positively oriented Seifert fibered space}

The following statement is a generalization of \cite[Theorem 3.4]{OSzSeifert} and is essentially the same as \cite[Theorem 8.9]{KMOS}.  We present a proof using the mapping cone formulas discussed in the previous section.

\begin{thm}\label{mainpos}
Let $K \subset S^3$ be a knot in the three-sphere, and suppose that there is a rational number $p/q>0$ and a negative definite or semi-definite graph $\Gamma$ with only one bad point with the property that $$S^3_{p/q}(K)\cong-\partial W(\Gamma),$$ then for any $k\in \Z$, all elements of $H_*(A_k^+)$ have even $\Z/2\Z$-grading.  Here, $A_k^+=C\{i\ge0 \text{ or }j\ge k\}$ denotes the chain complex associated to the knot $K\subset S^3$.

\end{thm}

\begin{cor}\label{postor}
Under the same conditions, all the elements of $HF^+_{\mathrm{red}}(S^3_0(K))$ have odd $\Z/2\Z$-grading, and all the torsion coefficients $t_i(K)$ (cf equation (\ref{torsion})) are non-negative.
\end{cor}

\begin{proof}
 By (\ref{zerosurgery}), $HF^+(S^3_0(K),i)$ is identified with the homology of the mapping cone of $v^+_i+h^+_i\co A^+_i\to B^+$.  When $i\neq 0$, we have $V_i\neq H_i$ (Lemma \ref{lem:V0H0}), so the map is surjective. Thus, the homology of the mapping cone comes entirely from the kernel of the map.  It then follows from Theorem \ref{mainpos} and the $\Z/2\Z$-grading-shift of the mapping cone that all the elements of $HF^+(S^3_0(K),i)$ have odd $\Z/2\Z$-grading.  When $i=0$, $V_0=H_0$; but the reduced part of the homology still comes from the kernel of the map, so the same argument applies to $HF^+_{\mathrm{red}}(S^3_0(K),0)$.  Finally, the statement of the torsion coefficients follows from the fact that
$$\chi(HF^+(S^3_0(K),i))=-t_i(K), \quad \text{for} \; i\neq 0$$ and
$$\chi(HF^+_{\mathrm{red}}(S^3_0(K),0))\geq-t_0(K)$$

\end{proof}

\begin{proof}[Proof of Theorem \ref{mainpos}]
By abusing of notation, we do not distinguish $A_k^+$ from its homology and write $$A_k^+=A_k^T \oplus A^{even}_{k,\mathrm{red}} \oplus A^{odd}_{k,\mathrm{red}},$$ where $A_k^T\cong\mathcal T^+$ has an even $\Z/2\Z$-grading, and the reduced group is divided into even and odd parts according to their $\Z/2\Z$-gradings.  We claim that $A^{odd}_{k,\mathrm{red}}=0$.  Otherwise, take $\theta \in A^{odd}_{k,\mathrm{red}}$ and extend it to $\mbox{\boldmath${\theta}$} \in  \mathbb A^{odd}_{i,\mathrm{red}}$ for some appropriate $i$ by adjoining zero in other relevant $A^+$ components.  Since the map $D_{i,p/q}^{T}\co\mathbb A_i^T\to \mathbb B_i^+$ is surjective (Lemma \ref{lem:>0Surj}), we can find $\mbox{\boldmath${\xi}$}\in \mathbb A_i^T$ such that $D_{i,p/q}^{T}(\mbox{\boldmath${\xi}$})=-D_{i,p/q}^+(\mbox{\boldmath${\theta}$})$. Thus $\mbox{\boldmath${\xi}$}+\mbox{\boldmath${\theta}$}$ is in the kernel of the mapping cone $$D_{i,p/q}^+\co\mathbb A_i^+\to \mathbb B_i^+$$
which, by Theorem \ref{OSmappingcone}, is isomorphic to $HF^+(S^3_{p/q}(K),i)$.   On the other hand, we see that $\mbox{\boldmath${\xi}$}+\mbox{\boldmath${\theta}$}$ has a mixing $\Z/2\Z$-grading.  But according to Theorem \ref{OSzPlumb}, the group $HF^+(S^3_{p/q}(K),i)$ is supported in even dimensions.  This is a contradiction!  Hence all elements of $A_k^+$, or more precisely, the homology $H_*(A_k^+)$ have even $\Z/2\Z$-grading.

\end{proof}

\subsection{Negatively oriented Seifert fibered space}

The case of negatively oriented Seifert fibered space is more interesting and subtle. First, we adapt the argument in the previous section to prove an analogous result for negatively Seifert fibered spaces.

\begin{prop}
Let $p/q>0$, and suppose that $S^3_{p/q}(K)$ is a negatively oriented Seifert fibered space.  Then for any $k\in \Z$, $$H_*(A_k^+)\cong \mathcal T^+_{even}\oplus A^{odd}_{k,\mathrm{red}}.$$
Here, $A_k^+=C\{i\ge0 \text{ or }j\ge k\}$ denotes the chain complex associated to the knot $K\subset S^3$.

\end{prop}

\begin{proof}
Again, by abusing of notation, we do not distinguish $A_k^+$ from its homology and write $$A_k^+=A_k^T \oplus A^{even}_{k,\mathrm{red}} \oplus A^{odd}_{k,\mathrm{red}},$$ where $A_k^T\cong\mathcal T^+$ has an even $\Z/2\Z$-grading, and the reduced group is divided into even and odd parts according to their $\Z/2\Z$-gradings.  This time we need to prove that $A^{even}_{k,\mathrm{red}}=0$.  If this were not true, take $\theta \in A^{even}_{k,\mathrm{red}}$ and extend it to $\mbox{\boldmath${\theta}$} \in  \mathbb A^{even}_{i,\mathrm{red}}$ for some appropriate $i$ by adjoining zero in other relevant $A^+$ components.  Since the map $D_{i,p/q}^{T}\co\mathbb A_i^T\to \mathbb B_i^+$ is surjective (Lemma \ref{lem:>0Surj}), we can find $\mbox{\boldmath${\xi}$}\in \mathbb A_i^T$ such that $D_{i,p/q}^{T}(\mbox{\boldmath${\xi}$})=-D_{i,p/q}^+(\mbox{\boldmath${\theta}$})$. Thus $\mbox{\boldmath${\xi}$}+\mbox{\boldmath${\theta}$}$ is in the kernel of the mapping cone $$D_{i,p/q}^+\co\mathbb A_i^+\to \mathbb B_i^+$$
which, by Theorem \ref{OSmappingcone}, is isomorphic to $HF^+(S^3_{p/q}(K),i)$.   Note that $\mbox{\boldmath${\xi}$}+\mbox{\boldmath${\theta}$}\in HF^+_{\mathrm{red}}(S^3_{p/q}(K),i)$ and it has an even $\Z/2\Z$-grading.  On the other hand, according to \cite[Proposition 2.5]{OSzAnn2}, if $Y$ is any rational homology three-sphere, then $HF^+_{\mathrm{red}}(Y)\cong   HF^+_{\mathrm{red}}(-Y)$, under a map which reverses the $\Z/2\Z$-grading.
Hence, there should not be any elements of $HF^+_{\mathrm{red}}(S^3_{p/q}(K),i)$ that are supported in even dimension.  This is a contradiction!
\end{proof}

We now apply (\ref{zerosurgery}), using the mapping cone $$v^+_i+h^+_i\co A^+_i\to B^+$$
 to compute $HF^+(S^3_0(K),i)$.  Note that $V_i<H_i$ when $i>0$, and the image of $A^{odd}_{i,\mathrm{red}}$ under $v_i$ and $h_i$ must be zero due to the $\Z/2\Z$-grading consideration.  Hence, we have the following:

\begin{prop}\label{groupofzerosurg}
Suppose $p/q>0$ and $S^3_{p/q}(K)$ is a negatively oriented Seifert fibered space.  Then for $i\neq 0$,  $$HF^+(S^3_0(K),i)\cong (A^{odd}_{i,\mathrm{red}})_{even} \oplus (\mathbb{F}_2)^{V_i}_{odd}.$$
and  $$HF^+(S^3_0(K),0)\cong (A^{odd}_{0,\mathrm{red}})_{even} \oplus \mathcal T^+ \oplus \mathcal T^+.$$

\end{prop}

Note the $\Z/2\Z$-grading shift in the statement of the above Proposition. This is due to Remark \ref{gradingshift}.  Applying Lemma \ref{lem:VHMono}, we can obtain a bound for the torsion coefficients $t_i(K)$.

\begin{cor}\label{negfourballgenus}
Suppose $p/q>0$ and $S^3_{p/q}(K)$ is a negatively oriented Seifert fibered space.  Then for $i> 0$, $$t_i(K)\leq \mathrm{max} \{g^*(K)-i, 0\}.$$
Moreover, if $g^*(K)<g$, then all the elements of $HF^+_{\mathrm{red}}(S^3_0(K),g-1)$ have even $\Z/2\Z$-grading.
\end{cor}

\begin{proof}
For $i\neq 0$,
\begin{align*}
t_i(K) &=-\chi(HF^+(S^3_0(K),i)\\
&=V_i-\rank (A^{odd}_{i,\mathrm{red}})\\
&\leq V_i\\
&\leq  \mathrm{max} \{g^*(K)-i, 0\}.
\end{align*}
The second part of the statement follows from Proposition \ref{groupofzerosurg}.
\end{proof}

When $0<p/q<3$, we can actually prove something stronger, which is a natural generalization of \cite[Proposition 3.5]{OSzSeifert}.

\begin{thm}\label{mainneg}
Suppose $0<p/q <3$ and $S^3_{p/q}(K)$ is a negatively oriented Seifert fibered space.  Then all the elements of $HF^+_{\mathrm{red}}(S^3_0(K),i)$ have even $\Z/2\Z$-grading.  %Moreover, $$d_{1/2}(S^3_0(K))=\frac{1}{2} \;\; \mathrm{and}\;\; d_{-1/2}(S^3_0(K))=-\frac{1}{2}.$$
Hence for all $i>0$, $t_i(K)\leq 0$.

\end{thm}

\begin{proof}
In light of Proposition \ref{groupofzerosurg}, it is enough to show that $V_i=0$ for all $i>0$; and since $V_i$ is a decreasing sequence, this is equivalent to $V_1=0$.  We prove by contradiction by assuming $V_1>0$.  Let us carefully elaborate on the cases $p/q=1 \;\mathrm{or}\; 2$, for all the other cases can be argued similarly.  Let $\theta\in A_1^T$ be the generator of $A_1^T\cong \mathcal{T}^+$.  Under the assumption that $V_1>0$, we have $h_1(\theta)=v_1(\theta)=0$.  Thus by adjoining 0 in other $A^+$, we can extend $\theta$ to $\mbox{\boldmath${\theta}$} \in \mathbb A^+_{1}$, and it is clear from the construction that $\mbox{\boldmath${\theta}$}\in H_*(\mathbb{X}^+_{i=1, p/q=1 \;\mathrm{or}\;2})\cong HF^+(S^3_{p/q}(K),i=1)$ has an even $\Z/2\Z$-grading.

We claim that $\mbox{\boldmath${\theta}$}$ is not in the free part $\mathcal{T}^+$ of $HF^+(S^3_{p/q}(K),i=1)$.  Otherwise, take $n>V_1$ and we have $U^{-n}\mbox{\boldmath${\theta}$}:=\mbox{\boldmath${\xi}$}=\{(s,\xi_s)\}_{s\in \Z}\in \mathbb{A}_1^T$.  Note that $\xi_0=U^{-n}\theta$.  For $p/q=2$, since $h_{-1}(\xi_{-1})=v_1(\xi_{0})$ and $H_{-1}=V_1$, we must have $\xi_{-1}=U^{-n}\theta_{-1}$ where $\theta_{-1}$ is the generator of $A^T_{-1}$.  For $p/q=1$, since $h_{0}(\xi_{-1})=v_1(\xi_{0})$ and $H_{0}=V_0\geq V_1$, we must have $\xi_{-1}=U^{-n-V_0+V_1}\theta_{0}$ where $\theta_{0}$ is the generator of $A^T_{0}$.  In either case, this implies that $\mbox{\boldmath${\theta}$}=U^n\mbox{\boldmath${\xi}$}$ has a non-zero component in $A^T_{-1}$ or $A^T_{0}$, contradicting to the original assumption of $\mbox{\boldmath${\theta}$}$.  Note that the same argument works for an arbitrary positive rational number $p/q <3 $, since what we really used in our argument is the fact that there is an $\mathbb{A}_i^+$ such that the left adjacent complex of $A_1^+$ is either $A_0^+$ or $A_{-1}^+$.  Hence, what we have achieved so far is finding an element $\mbox{\boldmath${\theta}$}$ with an even $\Z/2\Z$-grading that is not in the free part of $HF^+(S^3_{p/q}(K),i)$ for some $i$.  This contradicts to the fact that $S^3_{p/q}(K)$ is a negatively oriented Seifert fibered space, whose reduce part of Heegaard Floer homology should support on odd dimension.

\end{proof}

In fact, for an arbitrary positive rational number $p/q$, we can likewise obtain a bound for the torsion coefficients $t_i(K)$ in terms of the size of the surgery.

\begin{prop}\label{sizeofsurgery}
Suppose $0<p/q <2m+1$, $m\geq 1$, and $S^3_{p/q}(K)$ is a negatively oriented Seifert fibered space.  Then for $i>0$,
$$t_i(K)\leq \mathrm{max}\{m-i, 0\}.$$
Moreover, for $i\geq m$, all the elements of $HF^+_{\mathrm{red}}(S^3_0(K),i)$ have even $\Z/2\Z$-grading.
\end{prop}

\begin{proof}
Following the same argument of the proof for Theorem \ref{mainneg}, we are actually able to show that $V_m=0$ if $S^3_{p/q}(K)$ is a negatively oriented Seifert fibered space with $p/q<2m+1$.  Next, we apply Lemma \ref{lem:VHMono} and see that $V_i\leq m-i$ for $0<i<m$.  Thus \begin{align*}
t_i(K) &=-\chi(HF^+(S^3_0(K),i)\\
&=V_i-\rank (A^{odd}_{i,\mathrm{red}})\\
&\leq V_i\\
&\leq  \mathrm{max} \{m-i, 0\}.
\end{align*}
The second part of the statement follows from Proposition \ref{groupofzerosurg}.
\end{proof}

\subsection{Proofs of the various obstructions}

\begin{proof}[Proof of Theorem \ref{PSFS}]
This follows readily from Theorem \ref{mainpos} and Theorem \ref{mainneg}.

\end{proof}

\begin{proof}[Proof of Theorem \ref{knotobst}]
According to \cite[Corollary 4.5]{OSzKnot}, $$\HFK(S^3,K,g)\cong HF^+(S^3_0(K),g-1)$$ under an isomorphism which reverse parity when $g>1$.
Suppose that $S^3_{p/q}(K)$ is a positively oriented Seifert fibered space. It readily follows from Corollary \ref{postor} that $\HFK(S^3,K,g)$ is supported in even degrees and is non-trivial by equation (\ref{genus}).

Now, suppose $g>1$ and $S^3_{p/q}(K)$ is a negatively oriented Seifert fibered space.  If in addition, $2g-1 \geq p/q$; by plugging $m=g-1$ in Proposition \ref{sizeofsurgery}, we see that $HF^+(S^3_0(K),g-1)$ is supported in even degree.  Thus,  $\HFK(S^3,K,g)$ is supported in odd degree and is again non-trivial according to equation (\ref{genus}).

\end{proof}

\begin{proof}[Proof of Theorem \ref{fourballgenus}]
Suppose there is a $p/q$-surgery on $K$ that yields a Seifert fibered space.  By taking the mirror if necessary, we can assume that $p/q>0$.  According to Corollary \ref{surgcoeff}, $S^3_{p/q}(K)$ can never be a positively oriented Seifert fibered space.  On the other hand, we see from Corollary \ref{negfourballgenus} that all elements of $HF^+_{\mathrm{red}}(S^3_0(K),g-1)$ have even $\Z/2\Z$-grading.  Thus, $\HFK(S^3,K,g)$ is supported in odd degrees and has to be non-trivial according to equation (\ref{genus}).  It then follows from (\ref{Alexander}) that the Alexander polynomial must have non-zero leading term, which contradicts to the assumption that $\mathrm{deg}\;\Delta_K<g$.

\end{proof}

In order to prove Theorem \ref{sliceknot}, we need to find a relationship between $\chi(HF^+_{\mathrm{red}}(S^3_0(K),0)$ and $t_0(K)$.  This has to do with two numerical invariants analogous to the correction term, $d_{\pm 1/2}(Y_0)$, on a three-manifold $Y_0$ with first homology isomorphic to $\Z$.  More precisely, $d_{\pm 1/2}(Y_0)$ is the maximal $\Q$-grading of any element in $HF^+(Y_0)$ contained in the image of $HF^\infty(Y_0)$ whose parity is given by $\pm\frac{1}{2}+2\Z$.  We have the following identity:
$$
\chi(HF^+_{\mathrm{red}}(S^3_0(K),0))- (\frac{d_{-1/2}(S^3_0(K))-d_{1/2}(S^3_0(K))+1}{2})                     =-t_0(K)
$$
It also follows from the algebraic structure of $HF^\infty$ that
\begin{equation}\label{dineq}
d_{1/2}(Y_0)-1\leq d_{-1/2}(Y_0).
\end{equation}

\begin{lem}\label{dslice}
Suppose $K$ is a slice knot, then $d(S^3_1(K))=0$.
\end{lem}

\begin{proof}
By Lemma \ref{lem:VHMono}, $V_0=0$.  Now we can apply straightforwardly the formula from \cite[Proposition 2.11]{NiWu} for a general $p/q$ surgery:
$$ d(S^3_{p/q}(K),i)=d(L(p,q),i)-2\mathrm{max} \{V_{\lfloor \frac{i}{q}\rfloor}, H_{\lfloor \frac{i-p}{q}\rfloor} \}$$
and the result follows.
\end{proof}

\begin{proof}[Proof of Theorem \ref{sliceknot}]
First, we show that if $K$ is a slice knot, then
$$d_{1/2}(S^3_0(K))=\frac{1}{2} \;\;\;\; \mathrm{and} \;\;\;\; d_{-1/2}(S^3_0(K))=-\frac{1}{2}.$$
Indeed, the argument follows exactly the same line as that of \cite[Proposition 3.5]{OSzSeifert}: We start with the long exact sequence that connects the $HF^+$ of $S^3$, $S^3_0(K)$ and $S^3_1(K)$, and it is easy to see that $d_{-1/2}(S^3_0(K))=-1/2$.  This implies that $d_{1/2}(S^3_0(K))\leq 1/2$ by (\ref{dineq}). On the other hand, since the map from $HF^+(S^3_0(K), 0)$ to $HF^+(S^3_1(K))$ drops degree by 1/2, we see that $d(S^3_1(K))\leq d_{1/2}(S^3_0(K))-1/2$.  We apply Lemma \ref{dslice} to conclude that $d_{1/2}(S^3_0(K))\geq 1/2$.  Hence, $d_{1/2}(S^3_0(K))= 1/2$.

Next, by taking the mirror of $K$ if necessary, we need only to worry about positive surgery coefficients $p/q>0$.  It is then clear from Theorem \ref{PSFS} that no surgery on a knot with both positive and negative torsion coefficients can yield a positively oriented Seifert fibered space. So $S^3_{p/q}$ can only be a negatively oriented Seifert fibered space. But for $i>0$, it follows from Corollary \ref{negfourballgenus} that $t_i(K) \leq 0$; for $i=0$,
\begin{align*}
t_0(K)&=-\chi(HF^+_{\mathrm{red}}(S^3_0(K),0))+ (\frac{d_{-1/2}(S^3_0(K))-d_{1/2}(S^3_0(K))+1}{2})\\
&=-\chi(HF^+_{\mathrm{red}}(S^3_0(K),0))\\
&\leq 0
\end{align*}
where the inequality follows from Proposition \ref{groupofzerosurg}.  This contradicts to the original assumption on torsion coefficients and concludes the proof.

\end{proof}

\end{document}